\documentclass[11pt,twoside]{article}
\usepackage{}
\usepackage{undertilde}
\usepackage{amsmath}
\usepackage{amssymb}
\usepackage{fancyhdr}
\usepackage{latexsym}
\usepackage{bbding}
\usepackage{mathrsfs}
\usepackage{wasysym}
\usepackage{cite}
\usepackage{color}
\usepackage{multicol,graphics}

\newtheorem{theorem}{Theorem}[section]
\newtheorem{lemma}[theorem]{Lemma}
\newtheorem{corollary}[theorem]{Corollary}

\newtheorem{definition}[theorem]{Definition}
\newtheorem{remark}[theorem]{Remark}
\newtheorem{proposition}[theorem]{Proposition}
\numberwithin{equation}{section}
\newenvironment{proof}[1][Proof]{\noindent\textbf{#1.} }{\hfill $\Box$}
\allowdisplaybreaks

 \makeatletter
\begin{document}

\title{{Dynamics of a Nonlocal Dispersal SIS Epidemic Model}}
\author{Fei-Ying Yang and Wan-Tong Li\thanks{%
Corresponding author (wtli@lzu.edu.cn).} \\
School of Mathematics and Statistics,\\
Lanzhou University, Lanzhou, Gansu, 730000, P.R. China}
%
\maketitle

\begin{abstract}

This paper is concerned with a nonlocal dispersal susceptible-infected-susceptible (SIS) epidemic model with Dirichlet boundary condition, where the rates of disease transmission and recovery are assumed to be spatially heterogeneous. We introduce a basic reproduction number $R_0$ and establish threshold-type results on the global dynamic in terms of $R_0$. More specifically, we show that if the basic reproduction number is less than one, then the disease will be extinct, and if the basic reproduction number is larger than one, then the disease will persist. Particularly, our results imply that the nonlocal dispersal of the infected individuals may suppress the spread of the disease even though in a high-risk domain.



\end{abstract}


\section{Introduction}\label{Int}

\noindent

The susceptible-infected-susceptible (SIS) model is one of the basic compartment models to describe the transmission of communicable diseases in the filed of theoretical epidemiology. In the past years, an SIS epidemic reaction-diffusion model with Neumann boundary condition
\begin{equation*}\label{1101}
\begin{cases}
\frac{\partial S}{\partial t}=d_S\Delta S-\frac{\beta(x)SI}{S+I}
+\gamma(x)I ~&\text{in}~\Omega\times(0,+\infty),\\
\frac{\partial I}{\partial t}=d_I\Delta I+\frac{\beta(x)SI}{S+I}
-\gamma(x)I ~&\text{in}~\Omega\times(0,+\infty),\\
\partial_{\nu}S=\partial_{\nu}I=0 ~&\text{on}
~\partial\Omega\times(0,+\infty)
\end{cases}
\end{equation*}
has been proposed and studied by Allen et al. \cite{ABL-2008}.
Here, the functions $S(x,t)$ and $I(x,t)$ denote the density of susceptible and infected individuals at location $x$ and time $t$, respectively; $d_S$ and $d_I$ are positive diffusion rates for the susceptible and infected individuals; $\beta(x)$ and $\gamma(x)$ are positive H\"{o}lder continuous functions on $\bar{\Omega}$ that represent the rates of disease transmission and recovery at $x$. Meanwhile, the homogeneous Neumann boundary conditions mean that there is no population flux across the boundary $\partial\Omega$ and both the susceptible and infected individuals live in the self-contained environment. In \cite{ABL-2008}, Allen et al. mainly considered the impact of spatial heterogeneity of environment and movement of individuals on the persistence and extinction of a disease. After that, Peng and Liu \cite{Peng2009} further extended the results in \cite{ABL-2008}  and proved that the endemic equilibrium is globally asymptotically stable if it exists in some particular cases, which confirmed the conjecture proposed in \cite{ABL-2008}. Meanwhile, Peng \cite{Peng20091} provided a further understanding of the impacts of diffusion rates of the susceptible and infected population on the persistence and extinction of the disease. For more results of spatial heterogeneity of environment about the SIS epidemic models, one can see \cite{Peng2012,Peng2013} and the references therein.

In 2010, Huang et al. \cite{Huang2010} studied a diffusive SIS epidemic model with hostile boundary (Dirichlet) condition of the form
\begin{equation}\label{1102}
\begin{cases}
\frac{\partial S}{\partial t}=d_S\Delta S+\Lambda(x)-\frac{\beta(x)SI}{S+I}
+\gamma(x)I ~&\text{in}~\Omega\times(0,+\infty),\\
\frac{\partial I}{\partial t}=d_I\Delta I+\frac{\beta(x)SI}{S+I}
-\gamma(x)I ~&\text{in}~\Omega\times(0,+\infty),\\
S(x,t)=I(x,t)=0 ~&\text{on}
~\partial\Omega\times(0,+\infty).
\end{cases}
\end{equation}
In \eqref{1102}, the growth term $\Lambda(x)$ is used to balance the population decay in the boundary and makes the model more meaningful. They majored in considering the global stability of the disease-free equilibrium, the existence and uniqueness of a positive endemic steady, and also the global stability of the endemic steady for some particular cases.

It is well-known that the random dispersal which describes the local behavior of the movement to species, has been used to construct the population model in epidemiology and spatial ecology. However, Murray  \cite{Murray} pointed out that a local or short range diffusive flux proportional to the gradient is not sufficiently accurate to describe some biological phenomenon. In particular, Murray \cite[Chapter 17]{Murray} further emphasized the importance and intuitively necessity of the long range effects in the biological areas. Nowadays, the diffusion process is also described by an integral operators which represents the movement of species between nonadjacent spatial locations, such as $J*u-u=\int_{\mathbb{R}}J(x-y)u(y)dy-u$ in epidemiology and spatial ecology, see \cite{RN2012,Wang2002,HMMV2003,KLSH2010,KLSH2011,
YYLW2013,zhang-mcm,SLY2014,YLW2015,SLW2010,SLW2011,PLL2009} and the references therein.

In the present paper, we are concerned with the nonlocal counterpart of \eqref{1102} as follows
\begin{equation}\label{101}
\begin{cases}
\frac{\partial S}{\partial t}=d_S\int_{\mathbb{R}^N}J(x-y)(S(y,t)-S(x,t))dy
+\Lambda(x)-\frac{\beta(x)SI}{S+I}+\gamma(x)I,& x\in\Omega,~t>0,\\
\frac{\partial I}{\partial t}=d_I\int_{\mathbb{R}^N}J(x-y)(I(y,t)-I(x,t))dy
+\frac{\beta(x)SI}{S+I}-\gamma(x)I, & x\in\Omega,~t>0,\\
S(x,0)=S_0(x),~I(x,0)=I_0(x), & x\in\Omega,\\
S(x,t)=I(x,t)=0,& x\in{\mathbb{R}^N}\backslash\Omega,~t>0.\\
\end{cases}
\end{equation}
Here, $\Omega$ is a bounded domain, $\Lambda(x)$ is positive and continuous function on $\bar{\Omega}$ which represents the growth rate for the new born. $\beta(x)$ and $\gamma(x)$ are all positive continuous functions on $\bar{\Omega}$. Just as to say in \cite{ABL-2008} that the term $\frac{SI}{S+I}$ is a Lipschitz continuous function of $S$ and $I$ in the open first quadrant, we can extend its definition to the entire first quadrant by defining it to be zero when either $S=0$ or $I=0$. Throughout the whole paper, we always assume that the initial infected individuals are positive without other description, that is
\begin{equation}\label{102}
\int_{\Omega}I_0(x)dx>0,~\text{with}~S_0(x)\ge0~\text{and}
~I_0(x)\ge0~\text{for}~x\in\Omega,
\end{equation}
and the dispersal kernel function satisfies
\begin{equation}\label{103}
J(x)\in C(\bar{\Omega}),~J(0)>0,~J(x)=J(-x)\ge0,~\int_{\mathbb{R}^N}J(x)dx=1
~\text{and}~\int_{\Omega}J(x-y)dy\not\equiv1.
\end{equation}
Now, by the standard theory of semigroups of linear bounded operator \cite{Pazy}, we know from \cite{KLSH2010} that \eqref{101} admits a unique positive solution $(\tilde{S}(x,t),\tilde{I}(x,t))$, which is continuous with respect to $x$ and $t$.

On the other hand, we know that the basic reproduction number $R_0$ is an important threshold parameter to discuss the dynamic behavior of the epidemic model. For \eqref{1102}, $R_0$ can be defined by a variational formulation, that is
\begin{equation}\label{105}
R_0=\sup\left\{\frac{\int_{\Omega}\beta(x)\phi^2(x)dx}
{\int_{\Omega}[d_I|\nabla \phi(x)|^2+\gamma(x)\phi^2(x)]dx}:~\phi\in H_0^1(\Omega),~\phi\neq0\right\},
\end{equation}
and this definition strongly depends on the principal eigenvalue of the operator $d_I\Delta\cdot+\beta(x)-\gamma(x)$. Note that the nonlocal dispersal operator shares many similar properties with the random diffusion operator (see, e.g.,  \cite{AMRT2010,KLSH2010}). It is natural to ask whether the basic reproduction number can be defined by the same method as \eqref{105}. However, the answer is \textbf{negative}. The main reason is that the nonlocal operator
\begin{equation}\label{100}
L[u](x):=d_I\left(\int_{\Omega}J(x-y)u(y)dy-u(x)\right)
+\beta(x)u-\gamma(x)u
\end{equation}
does not admit principal eigenvalues in general, see \cite{HMMV2003,Coville2010,SHZH2010} and so on. To overcome this difficulty, we may define the basic reproduction number $R_0$ of system \eqref{101} by the abstract theory \cite{Thieme2009,WZ2012}. Particularly, we give the relations between $R_0$ and
\begin{equation}\label{110}
\mu_p(d_I):=\sup \left\{G(d_I)\Bigg|~u\in L^2(\Omega),~\int_{\Omega}u^2(x)dx=1 \right\},
\end{equation}
in which
$$G(d_I)=d_I\left(\int_{\Omega}\int_{\Omega}
J(x-y)\varphi(y)\varphi(x)dydx-\int_{\Omega}\varphi^2(x)dx
\right)+\int_{\Omega}(\beta(x)-\gamma(x))
\varphi^2(x)dx.$$
It follows from \cite{Coville2010,HMMV2003,SLY2014,SHZH2010,SLW2015} that $\mu_p(d_I)$ may not be the principal eigenvalue of the operator $L$. This is of independent interest.

Under the assumptions \eqref{102} and \eqref{103}, we can prove that the stationary problem of \eqref{101} admits a unique disease-free equilibrium and a unique endemic equilibrium. Further, the global stability of the disease-free equilibrium is discussed if $R_0<1$. Meanwhile, when $R_0>1$, and the susceptible and infected individuals have the same diffusion rate, then the global stability of the endemic equilibrium is also obtained. In particular, we have to emphasize that a lack of regularizing brings some technical difficulties to obtain our results. However, we can overcome by some auxiliary problems.

This paper is organized as follows. In Section 2, we give the definition of the basic production number of system \eqref{101}. Then, the existence, uniqueness and global attractivity of the disease-free equilibrium are obtained when $R_0<1$ in Section 3. In Section 4, we show the existence, uniqueness and global attractivity of the endemic equilibrium as $R_0>1$. Finally, a brief discussion is given to explain our results in the biological sense in Section 5.

\section{The basic reproduction number}
\noindent

In this section, we are concerned with the basic reproduction number of system \eqref{101}, which is an important threshold parameter in population models. Let $X=C(\bar{\Omega})$ be the Banach space of real continuous functions on $\bar{\Omega}$. Throughout this section, $X$ is considered as an ordered Banach space with a positive cone
$X_+=\{u\in X|~u\ge0\}$. It is well known that $X_+$ is generating, normal and has nonempty interior. Additionally, an operator $T: X\to X$ is called positive if $TX_+\subseteq X_+$.

\begin{definition}
A closed operator $\mathscr{A}$ in $X$ is called resolvent-positive, if the resolvent set of $\mathscr{A}$, $\rho(\mathscr{A})$, contains a ray $(\omega,\infty)$ and $(\lambda-\mathscr{A})^{-1}$ is a positive operator for all $\lambda>\omega$.
\end{definition}

We also define the spectral bound of $\mathscr{A}$ as follows
$$S(\mathscr{A})=\sup\{Re\lambda|~\lambda\in\sigma(\mathscr{A})\},$$
where $\sigma(\mathscr{A})$ denotes the spectrum of $\mathscr{A}$. The spectral radius of $\mathscr{A}$ is defined as
$$r(\mathscr{A})=\sup\{|\lambda|; ~\lambda\in\sigma(\mathscr{A})\}.$$
Below, we list some results from \cite{Thieme2009}.
\begin{theorem}\label{theorem203}
Let $\mathscr{A}$ be the generator of a $C_0-$semigroup $S$ on the ordered Banach space $X$ with a normal and generating cone $X_+$. Then, $\mathscr{A}$ is a resolvent-positive if and only if $S$ is a positive semigroup, i.e., $S(t)X_+\subset X_+$ for all $t\ge0$. If $\mathscr{A}$ is resolvent-positive, then
\begin{equation*}
(\lambda-\mathscr{A})^{-1}x=\lim_{b\to\infty}\int_{0}^{b}e^{\lambda t}S(t)xdt,~\lambda>S(\mathscr{A}),~x\in X.
\end{equation*}
\end{theorem}
\begin{theorem}\label{theorem204}
Let $\mathscr{B}$ be a resolvent-positive operator on $X$, $S(\mathscr{B})<0$ and $\mathscr{A}=\mathscr{C}+\mathscr{B}$ a positive perturbation of $\mathscr{B}$. If $\mathscr{A}$ is resolvent-positive, $S(\mathscr{A})$ has the same sign as $r(-\mathscr{C}\mathscr{B}^{-1})-1$.
\end{theorem}

Define
\begin{equation}\label{202}
A[u](x):=d_I\left(\int_{\Omega}J(x-y)u(y)dy-u(x)\right)
-\gamma(x)u(x).
\end{equation}
Following from \cite{KLSH2010,Pazy}, we know that the operators $L$ ($L$ is defined as \eqref{100}) and $A$ can respectively generate a positive $C_0$-semigroup on $X$. Thus, according to Theorem \ref{theorem203}, we know that the operators $L$ and $A$ are all resolvent-positive operators. Meanwhile, we have the following result.
\begin{proposition}\label{pro205}
If the operator $A$ is defined by \eqref{202}, then $S(A)<0$.
\end{proposition}
\begin{proof}
Define
\begin{equation*}
\mu_p=\sup_{\scriptstyle \varphi\in L^2(\Omega) \atop\scriptstyle \varphi\neq0
}\frac{d_I\left(\int_{\Omega}\int_{\Omega}
J(x-y)\varphi(y)\varphi(x)dydx-\int_{\Omega}\varphi^2(x)dx\right)
-\int_{\Omega}\gamma(x)
\varphi^2(x)dx}{\int_{\Omega}\varphi^2(x)dx}.
\end{equation*}
Due to
\begin{equation*}
\begin{aligned}
&\int_{\Omega}\int_{\Omega}
J(x-y)\varphi(y)\varphi(x)dydx-\int_{\Omega}\varphi^2(x)dx\\
\le& -\frac12\int_{\Omega}\int_{\Omega}
J(x-y)(\varphi(y)-\varphi(x))^2dydx\le0,
\end{aligned}
\end{equation*}
we have $\mu_p<0$. Since $\gamma(x)\in C(\bar{\Omega})$, there exists some $x_0\in\bar{\Omega}$ such that $\gamma(x_0)=\min\limits_{x\in\bar{\Omega}}\gamma(x)$. Define a function sequence as follows:
\begin{equation*}
\gamma_n(x)=
\begin{cases}
\gamma(x_0),&~x\in B_{x_0}(\frac1n),\\
\gamma_{n,1}(x),&~x\in(B_{x_0}(\frac2n)\backslash B_{x_0}(\frac1n)),\\
\gamma(x),&~x\in\Omega\backslash B_{x_0}(\frac2n),
\end{cases}
\end{equation*}
where $B_{x_0}(\frac1n)=\{x\in\Omega|~|x-x_0|<\frac1n\}$, $\gamma_{n,1}(x)$ satisfies $\gamma_{n,1}(x)\ge \gamma(x_0)$, and $\gamma_{n,1}(x)$ is continuous in $\Omega$. Indeed, $\gamma_{n,1}(x)$ exists if only we take $n$ is large enough, denoted by $n\ge n_0>0$. By the construction of function sequences $\{\gamma_n(x)\}_{n=1}^{\infty}$, we know that $\|\gamma_n-\gamma\|_{L^\infty(\Omega)}\to0$ as $n\to+\infty$ and the eigenvalue problem
\begin{equation*}
A_n[\varphi](x):=d_I\left(\int_{\Omega}
J(x-y)\varphi(y)dy-\varphi(x)\right)-\gamma_n(x)\varphi(x)
=\mu\varphi(x)~~~\text{in}~\Omega
\end{equation*}
admits a principal eigenpair $(\mu_p^n, \varphi_n(x))$, where $\mu_p^n\to\mu_p$ as $n\to+\infty$. Note that $\mu_p^n=S(A_n)$ for each given $n$ (see Bates and Zhao \cite{BZH2007}). Since $\mu_p<0$, there exists some $\delta>0$ such that $\mu_p^n<-\delta$ provided $n\ge n_0$ for some $n_0>0$. Thus, we have $S(A_n)<-\delta$ for $n\ge n_0$. Due to $\gamma_n\to \gamma$ as $n\to+\infty$, we can obtain that $S(A_n)\to S(A)$ as $n\to+\infty$, see \cite[Lemma 3.1]{SZ2012}. This implies that $S(A)<0$. The proof is complete.
\end{proof}

Now, consider the nonlocal dispersal problem
\begin{equation*}
\frac{\partial u_I(x,t)}{\partial t}=d_I\left(\int_{\Omega}J(x-y)u_I(y,t)dy-u_I(x,t)\right)
-\gamma(x)u_I(x,t),
\end{equation*}
where $x\in\Omega$ and $t>0$. If $u_I(x,t)$ is thought of as a density of the infected individuals at a point $x$ at time $t$, $J(x-y)$ is thought of as the probability distribution of jumping from location $y$ to location $x$, then $\int_{\mathbb{R}^N}J(y-x)u(y,t)dy$ is the rate at which the infected individuals are arriving at position $x$ from all other places, and $-\int_{\mathbb{R}^N}J(y-x)u(x,t)dy$ is the rate at which they are leaving location $x$ to travel to all other sites.
By the theory of semigroups of linear operators, we know that the operator $A$ can generate a uniformly continuous semigroup, denoted by $T(t)$. Let $\phi(x)$ be the distribution of initial infection. Then, $T(t)\phi(x)$ is the distribution of the infective members at time $t$ under the synthetical influences of mobility and transfer of individuals in infected compartments.
Inspired by the work in \cite{WZ2012,Diekmann1990,Driessche2002,WZ2011}, we may define the spectral radius of $\mathscr{L}$ as the basic reproduction number of system \eqref{101}, that is $R_0=r(\mathscr{L})$, where
\begin{equation*}
\mathscr{L}[\phi](x):=\beta(x)\int_{0}^{\infty}T(t)\phi dt.
\end{equation*}
Let $F[\psi](x)=\beta(x)\psi(x)$ for $\psi\in X$. Then, we have the following result.
\begin{theorem}\label{theorem206}
$R_0-1$ has the same sign as $\lambda_*:=S(A+F)$.
\end{theorem}
\begin{proof}
Since $A$ is the generator of the semigroup $T(t)$ on $X$ and $A$ is resolvent-positive, it then follows from Theorem \ref{theorem203} that
\begin{equation}\label{203}
(\lambda I-A)^{-1}\phi=\int_{0}^{\infty}e^{-\lambda t}T(t)\phi dt~\text{for}~\lambda>S(A),~\phi\in X.
\end{equation}
Choosing $\lambda=0$ in \eqref{203}, we obtain
\begin{equation}
-A^{-1}\phi=\int_{0}^{\infty}T(t)\phi dt~\text{for all}~\phi\in X.
\end{equation}
Then, the definition of operator $\mathscr{L}$ implies that $\mathscr{L}=-FA^{-1}$. Note that $L=A+F$. Since $L$ is a resolvent-positive operator and $S(A)<0$, following from Theorem \ref{theorem204}, we have $S(L)$ has the same sign as $r(\mathscr{L})-1=R_0-1$. The proof is complete.
\end{proof}

Consider the eigenvalue problem
\begin{equation}\label{204}
L[u](x)=\mu u(x)~~~~\text{in}~~\Omega.
\end{equation}
It is well known that if $\mu_p(d_I)$ is the principal eigenvalue of \eqref{204}, then $\mu_p(d_I)=S(A+F)$, see
\cite{BZH2007}. In this case, $\mu_p(d_I)$ has the same sign as $R_0-1$ according to Theorem \ref{theorem206}. However, we still have the following result.
\begin{corollary}\label{remark206}
$\mu_p(d_I)$ has the same sign as $R_0-1$.
\end{corollary}

This conclusion can be proved by the same arguments as Proposition \ref{pro205} and Theorem \ref{theorem206}. In fact, $\mu_p(d_I)=S(A+F)$ whether $\mu_p(d_I)$ is the principal eigenvalue of \eqref{204} or not.

Next, we intend to discuss the effect of $\beta(x),\gamma(x)$ on the basic reproduction number $R_0$. Consider the eigenvalue problem
\begin{equation}\label{305}
\begin{cases}
\int_{\mathbb{R}^N}J(x-y)(\varphi(y)-\varphi(x))dy=-\lambda \varphi(x)&~~\text{in}~\Omega,\\
\varphi(x)=0&~~\text{on}~\mathbb{R}^N\backslash\Omega,
\end{cases}
\end{equation}
which has been studied by Garc\'{i}a-Meli\'{a}n and Rossi \cite{MR2009}.
\begin{lemma}\label{lemma302}
The eigenvalue problem \eqref{305} admits a unique principal eigenvalue $\lambda_1$ and its corresponding eigenfunction $\varphi(x)$ is of class $C(\bar{\Omega})$. Moreover, $0<\lambda_1<1$ and
\begin{equation*}
\lambda_1=\inf_{\psi\in L^2(\Omega),\psi\neq0}
\frac{\int_{\Omega}\psi^2(x)dx-\int_{\Omega}\int_{\Omega}
J(x-y)\psi(y)\psi(x)dydx}{\int_{\Omega}\psi^2(x)dx}.
\end{equation*}
\end{lemma}

\begin{corollary}\label{cor2}
Assume $\beta(x_0)>\gamma(x_0)$ for some $x_0\in\Omega$. Then there is some $\hat{d}_I>0$ such that $R_0>1$ for $0<d_I<\hat{d}_I$ and $R_0<1$ for $d_I>\hat{d}_I$. Further, if $\beta(x)<\gamma(x)$ for all $x\in\bar{\Omega}$, then $R_0<1$ for all $d_I>0$.
\end{corollary}
\begin{proof}
Let $m(x)=\beta(x)-\gamma(x)$ and denote $\mu_p(d_I,m):=\mu_p(d_I)$. By the definition of $\mu_p(d_I)$, it is easy to see that $\mu_p(d_I,m)$ is non-increasing on any $d_I>0$. Meanwhile, $\mu_p(d_I,m)$ is continuous on $d_I$. In fact, for any $\delta>0$, the simple calculation yields that
\begin{eqnarray*}
\mu_p(d_I+\delta,m)&=& \sup \left\{G(d_I+\delta)\Bigg|~u\in L^2(\Omega),~\int_{\Omega}u^2(x)dx=1 \right\}\\
&\le& \sup \left\{G(d_I)\Bigg|~u\in L^2(\Omega),~\int_{\Omega}u^2(x)dx=1 \right\}\\
&& +\delta \sup_{\psi\in L^2(\Omega),\int_{\Omega}\psi^2(x)dx=1}
\left\{\int_{\Omega}\int_{\Omega}
J(x-y)\psi(y)\psi(x)dydx-\int_{\Omega}\psi^2(x)dx\right\}\\
&=& \mu_p(d_I,m)-\delta\lambda_1.
\end{eqnarray*}
On the other hand, for any $\delta>0$, we have
\begin{eqnarray*}
\mu_p(d_I,m)-\delta\lambda_1 &=& \sup \left\{G(d_I)\Bigg|~u\in L^2(\Omega),~\int_{\Omega}u^2(x)dx=1 \right\}-\delta\lambda_1\\
&\le& \mu_p(d_I+\delta,m)+\delta\lambda_1-\delta\lambda_1
=\mu_p(d_I+\delta,m).
\end{eqnarray*}
Thus, for any $\delta>0$, there is
\begin{equation*}
|\mu_p(d_I+\delta,m)-\mu_p(d_I,m)|\le\delta.
\end{equation*}

Additionally, we claim that $\mu_p(0,m)>0$ if $m(x_0)>0$ for some $x_0\in\Omega$. Since $m(x)\in C(\bar{\Omega})$ and $m(x_0)>0$, there exists some ball $B_\rho(x_0)$ with center $x_0$ and radius $\rho$ such that
$m(x)>0$ for all $x\in B_\rho(x_0)$. Set $\Omega_*=B_\rho(x_0)\cap\Omega$. Then, $\Omega_*\neq\emptyset$. Now, choose some function $\varphi\in L^2(\Omega)$ satisfying $\varphi>0$ in $\Omega_*$ and $\varphi=0$ in $\Omega\backslash\Omega_*$ as a test function. Hence, it is easily seen from the definition of $\mu_p(d_I,m)$ that $\mu_p(0,m)>0$ and the claim holds. Analogously, the same discussion can give that $\mu_p(0,m)<0$ if $\beta(x)<\gamma(x)$ for all $x\in\bar{\Omega}$.

Finally, we have $\mu_p(d_I,m)\to-\infty$ as $d_I\to+\infty$. Indeed, there is
\begin{equation*}
\mu_p(d_I,m)=d_I\mu_p\left(1,\frac{m}{d_I}\right).
\end{equation*}
By the continuity of $\mu_p(d_I,m)$ on $m$, we have
\begin{equation*}
\mu_p\left(1,\frac{m}{d_I}\right)\to \mu_p(1,0)=-\lambda_1<0~~\text{as}~d_I\to+\infty.
\end{equation*}
Thus, there is some $d_0>0$ such that for any $d\ge d_0$
\begin{equation*}
\mu_p(d_I,m)\le\frac{d_I}{2}\mu_p(1,0)=-\frac{d_I}{2}\lambda_1.
\end{equation*}
Consequently, we obtain $\mu_p(d_I,m)\to-\infty$ as $d_I\to+\infty$.

Now, by the discussion as above, if there is some $x_0\in\Omega$ such that $m(x_0)>0$, then there must be some $\hat{d}_I$ so that $\mu_p(d_I)>0$ for $0<d_I<\hat{d}_I$ and $\mu_p(d_I)<0$ for $d_I>\hat{d}_I$. And if $m(x)<0$ for all $x\in\bar{\Omega}$, it must be $\mu_p(d_I)<0$ for all $d_I>0$. Consequently, following Corollary \ref{remark206}, we can complete our proof.
\end{proof}

\begin{corollary}\label{cor3}
Suppose $\int_{\Omega}\beta(x)dx>\int_{\Omega}\gamma(x)dx$.
Then there is some $\hat{d}_I>0$ such that $R_0>1$ for $0<d_I<\hat{d}_I$ and $R_0<1$ for $d_I>\hat{d}_I$.
\end{corollary}

This corollary can be easily obtained by the definition of $\mu_p(d_I)$ and the analysis as above.

\section{The disease-free equilibrium}
\noindent

In this section, we consider the existence, uniqueness and global stability of the disease-free equilibrium of \eqref{101}.
\begin{definition}\label{def2}
We say that a steady state $(S_*(x),I_*(x))$ of system \eqref{101} is globally stable if the solutions $(\tilde{S}(x,t),\tilde{I}(x,t))$ of \eqref{101} satisfy
\begin{equation*}
\lim_{t\to+\infty}(\tilde{S}(x,t),\tilde{I}(x,t))=
(S_*(x),I_*(x))
\end{equation*}
for any initial data $(S_0(x),I_0(x))$ that satisfies $S_0(x),I_0(x)>0$ in $\Omega$ and $S_0(x),I_0(x)\in C(\bar{\Omega})$.
\end{definition}

We first consider the following problem
\begin{equation}\label{301}
\begin{cases}
d_S\int_{\mathbb{R}^N}J(x-y)(u(y)-u(x))dy+\Lambda(x)=0&~~\text{in}~\Omega,\\
u(x)=0&~~\text{on}~\mathbb{R}^N\backslash\Omega.
\end{cases}
\end{equation}

\begin{lemma}\label{lemma301}
There is a unique positive solution $u^*(x)$ of \eqref{301} and $u^*(\cdot)\in C(\bar{\Omega})$.
\end{lemma}
\begin{proof}
Let $\varphi(x)$ be the corresponding eigenfunction associated to $\lambda_1$. Construct $\underline{u}(x)=\varepsilon\varphi(x)$ and $\overline{u}(x)=M\varphi(x)$ for some $\varepsilon>0$ with $M=1+\frac{\max_{\bar{\Omega}}\Lambda(x)}
{\lambda_1d_S\min_{\bar{\Omega}}\varphi(x)}$. Note that
\begin{equation*}
\begin{aligned}
&d_S\int_{\mathbb{R}^N}J(x-y)(\underline{u}(y)-\underline{u}(x))dy
+\Lambda(x)\\
=& -\lambda_1\varepsilon d_S\varphi(x)+\Lambda(x)\ge0
\end{aligned}
\end{equation*}
for sufficiently small $\varepsilon$. Meanwhile, by the definition of $M$, we have
\begin{equation*}
\begin{aligned}
&d_S\int_{\mathbb{R}^N}J(x-y)(\overline{u}(y)-\overline{u}(x))dy
+\Lambda(x)\\
=& -\lambda_1M d_S\varphi(x)+\Lambda(x)\le0.
\end{aligned}
\end{equation*}
Thus, the standard iterative method implies that \eqref{301} admits a positive solution $\underline{u}(x)\le u^*(x)\le\overline{u}(x)$. Thus, $u^*(x)$ satisfies
\begin{equation*}
u^*(x)=\int_{\Omega}J(x-y)u^*(y)dy+\frac{\Lambda(x)}{d_S}
\end{equation*}
and $u^*(x)\in C(\bar{\Omega})$.
Assume $\tilde{u}(x)$ is another positive solution of \eqref{301}. Let $w(x)=u^*(x)-\tilde{u}(x)$. Then, $w(x)$ satisfies
\begin{equation*}
\begin{cases}
\int_{\mathbb{R}^N}J(x-y)(w(y)-w(x))dy=0 &~~~\text{in}~\Omega,\\
w(x)=0 &~~~\text{on}~\mathbb{R}^N\backslash\Omega.
\end{cases}
\end{equation*}
According to Proposition 2.2 in \cite{AMRT2010}, we know $w(x)\equiv0$. That is $u^*(x)=\tilde{u}(x)$. This implies that the positive solution of \eqref{301} is unique. The proof is complete.
\end{proof}
\begin{corollary}
There is a unique disease-free equilibrium  $(S_*(x),0)$ of \eqref{101} and $S_*(x)\in C(\bar{\Omega})$.
\end{corollary}

Now, we have the following global stability result.

\begin{theorem}\label{theorem302}
If $R_0<1$, then the positive solutions of \eqref{101} converge to the disease-free equilibrium $(S_*(x),0)$ uniformly on $x$ as $t\to+\infty$.
\end{theorem}
\begin{proof}
Since $R_0<1$, we have $\mu_p(d_I)<0$ following from Theorem \ref{theorem206} and Corollary \ref{remark206}.
First, consider the linear problem
\begin{equation}\label{3033}
\frac{\partial w(x,t)}{\partial t}=d_I\int_{\mathbb{R}^N}J(x-y)
(w(y,t)-w(x,t))dy
+(\beta(x)-\gamma(x))w(x,t)
\end{equation}
with initial value $w(x,0)= I_0(x)$ for $x\in\Omega$, and
the boundary condition
$w(x,t)=0$ for $x\in\mathbb{R}^N\backslash\Omega, t>0$.
Then, the semigroup theory implies that \eqref{3033} admits a unique positive solution $w(x,t)$.
Let $m(x)=\beta(x)-\gamma(x)$. Since $m(x)\in C(\bar{\Omega})$, we can find a function sequence $\{m_n(x)\}_{n=1}^{\infty}$ with $\|m_n-m\|_{L^\infty}\to0$ as $n\to+\infty$ such that the eigenvalue problem
\begin{equation*}
d_I\left(\int_{\Omega}J(x-y)u(y)dy-u(x)\right)
+m_n(x)u(x)=\mu u(x) ~~~\text{in}~\Omega
\end{equation*}
admits a principal eigenpair $(\mu_p^n(d_I), \varphi_n(x))$. Meanwhile, there is $\mu_p^n(d_I)\to\mu_p(d_I)$ as $n\to+\infty$.
Taking $n$ large enough, provided $n\ge n_0>0$, we have
\begin{equation*}
\mu_p^n(d_I)\le\frac12\mu_p(d_I)-\|m_n-m\|_{L^\infty(\Omega)}.
\end{equation*}
Construct $\overline{w}(x,t)=le^{\frac13\mu_p(d_I)t}\varphi_n(x)$ and normalize $\varphi_n$ as $\|\varphi_n\|_{L^\infty(\Omega)}=1$. Then, the direct computation yields that
\begin{eqnarray*}
&& \frac{\partial\overline{w}(x,t)}{\partial t}-d_I\left(\int_{\Omega}J(x-y)\overline{w}(y,t)dy
-\overline{w}(x,t)\right)-(\beta(x)-\gamma(x))\overline{w}(x,t)\\
&=& \frac13\mu_p(d_I)le^{\frac13\mu_p(d_I)t}\varphi_n(x)
+le^{\frac13\mu_p(d_I)t}(m_n(x)-m(x))\varphi_n(x)\\
&&-le^{\frac13\mu_p(d_I)t}\left[d_I\left(\int_{\Omega}J(x-y)\varphi_n(y)dy
-\varphi_n(x)\right)+m_n(x)\varphi_n(x)\right]\\
&\ge& \left[-\frac16\mu_p(d_I)+\|m_n-m\|_{L^\infty(\Omega)}+(m_n(x)-m(x))\right]
le^{\frac13\mu_p(d_I)t}\varphi_n(x)\ge0,
\end{eqnarray*}
provided $n\ge n_0$. Now, taking $l$ large enough such that $l\varphi_n(x)\ge I_0(x)$, we have $w(x,t)\le\overline{w}(x,t)$ for all $x\in\Omega$ and $t>0$ by comparison principle. Note that \begin{equation*}
\begin{aligned}
& \frac{\partial w(x,t)}{\partial t}-d_I\left(\int_{\Omega}J(x-y)w(y,t)dy
-w(x,t)\right)-\frac{\beta(x)w(x,t)\tilde{S}(x,t)}{w(x,t)+\tilde{S}(x,t)}
+\gamma(x)w(x,t)\\
\ge& \frac{\partial w(x,t)}{\partial t}-d_I\left(\int_{\Omega}J(x-y)w(y,t)dy
-w(x,t)\right)-(\beta(x)-\gamma(x))w(x,t)=0.
\end{aligned}
\end{equation*}
Thus, $\tilde{I}(x,t)\le w(x,t)$ for all $x\in\Omega$ and $t\ge0$. It then follows that
\begin{equation*}
\tilde{I}(x,t)\to0~~\text{uniformly on}~\bar{\Omega}~\text{as}~t\to+\infty.
\end{equation*}

Next, we want to show that $\tilde{S}(x,t)\to S_*(x)$ uniformly on $\bar{\Omega}$ as $t\to+\infty$. Set $v(x,t)=\tilde{S}(x,t)-S_*(x)$. Then, there is
\begin{equation}\label{303}
\frac{\partial v(x,t)}{\partial t}=d_S\left(\int_{\Omega}J(x-y)v(y,t)dy-v(x,t)\right)
+\gamma(x)\tilde{I}
-\frac{\beta(x)\tilde{S}\tilde{I}}{\tilde{S}+\tilde{I}},~~x\in\Omega.
\end{equation}
Since $\beta(x),\gamma(x)\in C(\bar{\Omega})$, by the above argument, there exists some constant $c_0>0$ such that
\begin{equation*}
\left|\gamma\tilde{I}
-\frac{\beta\tilde{S}\tilde{I}}{\tilde{S}+\tilde{I}}\right|\le c_0e^{\frac13\mu_p(d_I)t}.
\end{equation*}
Let $W(t)=\int_{\Omega}v^2(x,t)dx$. Then, we have
\begin{eqnarray*}
\frac{dW(t)}{dt}&=& 2\int_{\Omega}v(x,t)\frac{\partial v(x,t)}{\partial t}dx\\
&=& 2\int_{\Omega}v(x,t)\left[d_S\left(\int_{\Omega}J(x-y)v(y,t)dy-v(x,t)\right)
+\gamma(x)\tilde{I}
-\frac{\beta(x)\tilde{S}\tilde{I}}{\tilde{S}+\tilde{I}}\right]dx\\
&=& 2d_S\left(\int_{\Omega}\int_{\Omega}J(x-y)v(y,t)v(x,t)dydx
-\int_{\Omega}v^2(x,t)dx\right)\\
&& +2\int_{\Omega}\left(\gamma(x)\tilde{I}
-\frac{\beta(x)\tilde{S}\tilde{I}}{\tilde{S}
+\tilde{I}}\right)v(x,t)dx\\
&\le& -2d_S\lambda_1W(t)+2c_0|\Omega|^{\frac12}
e^{\frac13\mu_p(d_I)t}W^{\frac12}(t).
\end{eqnarray*}
Denote $a=-2d_S\lambda_1$, $b=2c_0|\Omega|^{\frac12}$ and $c=\frac13\mu_p(d_I)$. Then, the direct calculation yields that
\begin{equation*}
W(t)\le
\begin{cases}
(l_1+2bt)^2e^{at}~~&\text{if}~c=\frac a2,\\
\left(\frac{2b}{c-\frac a2}e^{ct}+l_2e^{\frac a2t}\right)^2~~&\text{if}~c\neq\frac a2,
\end{cases}
\end{equation*}
in which $l_1=W^{\frac 12}(0)$ if $c=\frac a2$ and $l_2=W^{\frac 12}(0)-\frac{2b}{c-\frac a2}$ if $c\neq\frac a2$. Thus, we have
\begin{equation*}
\|v(\cdot,t)\|_{L^2(\Omega)}\le
\begin{cases}
(k_1+k_2t)e^{\frac a2t}~~&\text{if}~c=\frac a2,\\
k_3e^{ct}+k_4e^{\frac a2t}~~&\text{if}~c\neq\frac a2
\end{cases}
\end{equation*}
for some positive constants $k_i (i=1,2,3,4)$.
It follows from \eqref{303} that
\begin{equation*}
v(x,t)=v_0(x)e^{-d_St}+e^{-d_St}\int_0^te^{d_Ss}
\left(d_S\int_{\Omega}J(x-y)v(y,s)dy+\gamma\tilde{I}
-\frac{\beta\tilde{S}\tilde{I}}{\tilde{S}
+\tilde{I}}\right)ds.
\end{equation*}
By H\"{o}lder inequality, we have
\begin{equation*}
\int_{\Omega}J(x-y)v(y,s)dy\le c\|v(\cdot,s)\|_{L^2(\Omega)}
\end{equation*}
for some positive constant $c$. Thus, there is
\begin{equation}\label{306}
|v(x,t)|\le
\begin{cases}
\tilde{k}_1e^{-d_St}+(\tilde{k}_2t+\tilde{k}_3)e^{\frac a2t}+\tilde{k}_4e^{ct}
~~&\text{if}~c=\frac a2,\\
\tilde{k}_5e^{-d_St}+\tilde{k}_6e^{\frac a2t}+\tilde{k}_7e^{ct}
~~&\text{if}~c\neq\frac a2,
\end{cases}
\end{equation}
where $\tilde{k}_i$ $(i=1,\cdot\cdot\cdot,7)$ are some positive constants. Since $v(x,t)\in C(\bar{\Omega}\times(0,\infty))$, we obtain that
\begin{equation*}
v(x,t)\to0~~\text{uniformly on}~\bar{\Omega}~\text{as}~t\to+\infty.
\end{equation*}
We complete the proof.
\end{proof}

\section{The endemic equilibrium}
\noindent

In this section, we consider the existence, uniqueness and global stability of the endemic equilibrium of \eqref{101}. To these goals, we study the steady state problem of \eqref{101}
\begin{equation}\label{401}
\begin{cases}
d_S\int_{\mathbb{R}^N}J(x-y)(S(y)-S(x))dy=\frac{\beta(x)SI}{S+I}
-\gamma(x)I-\Lambda(x), &x\in\Omega,\\
d_I\int_{\mathbb{R}^N}J(x-y)(I(y)-I(x))dy=-\frac{\beta(x)SI}{S+I}
+\gamma(x)I, &x\in\Omega,\\
I(x)=S(x)=0, &x\in\mathbb{R}^N\backslash\Omega.
\end{cases}
\end{equation}
By adding the two equations of \eqref{401}, we obtain the equivalent system
\begin{equation}\label{402}
\begin{cases}
\int_{\mathbb{R}^N}J(x-y)[(d_SS(y)+d_II(y))
-(d_SS(x)+d_II(x))]dy+\Lambda(x)=0, &x\in\Omega,\\
d_I\int_{\mathbb{R}^N}J(x-y)(I(y)-I(x))dy=-\frac{\beta(x)SI}{S+I}
+\gamma(x)I, &x\in\Omega,\\
I(x)=S(x)=0, &x\in\mathbb{R}^N\backslash\Omega.
\end{cases}
\end{equation}
Combining \eqref{301} and the first equation of \eqref{402}, we can get $d_SS(x)+d_II(x)=d_SS_*(x)$ for $x\in\Omega$. This yields that
\begin{equation*}
S(x)=\frac{d_SS_*(x)-d_II(x)}{d_S}.
\end{equation*}
Substituting this $S(x)$ into the second equation of \eqref{402}, we get
\begin{equation}\label{403}
d_I\left(\int_{\mathbb{R}^N}J(x-y)I(y)dy-I(x)\right)+\left[
\beta(x)-\gamma(x)-\frac{d_S\beta(x)I}
{d_SS_*(x)+(d_S-d_I)I}\right]I=0
\end{equation}
for $x\in\Omega$ and $I(x)=0$ for $x\in\mathbb{R}^N\backslash\Omega$.
\begin{theorem}\label{theorem401}
Assume $R_0>1$. Then, there is a unique positive solution $(S^*,I^*)$ of \eqref{401} satisfying $S^*,I^*\in C(\bar{\Omega})$ and $S^*>0, I^*>0$. Moreover, $0<I^*(x)<\frac{d_S}{d_I}S_*(x)$ for $x\in\Omega$.
\end{theorem}
\begin{proof}
Since $R_0>1$, we have $\mu_p(d_I)>0$ according to Corollary \ref{remark206}. Let $m(x)=\beta(x)-\gamma(x)$. By the same arguments as in Theorem \ref{theorem302}, we can find a function sequence $\{m_n(x)\}_{n=1}^{\infty}$ such that $\|m_n-m\|_{L^\infty(\Omega)}\to0$ as $n\to+\infty$ and the eigenvalue problem
\begin{equation*}
\begin{cases}
d_I\int_{\mathbb{R}^N}J(x-y)(\varphi(y)-\varphi(x))dy+m_n(x)\varphi(x)=\mu \varphi(x)~~~&\text{in}~\Omega,\\
\varphi(x)=0~~~&\text{on}~\mathbb{R}^N\backslash\Omega
\end{cases}
\end{equation*}
admits a principal eigenpair $(\mu_p^n(d_I), \varphi_n(x))$, and $\mu_p^n(d_I)$ satisfies
\begin{equation*}
\mu_p^n(d_I)\ge \frac12\mu_p(d_I)+\|m_n-m\|_{L^\infty(\Omega)},
\end{equation*}
provided $n\ge n_0$ for large enough $n_0$. Normalize $\varphi_n$ by $\|\varphi_n\|_{L^\infty(\Omega)}=1$ and construct $\underline{I}(x)=\varepsilon\varphi_n(x)$ for $\varepsilon>0$, $\overline{I}(x)=\frac{d_S}{d_I}S_*(x)$. Then, $\underline{I}(x)$ satisfies
\begin{eqnarray*}
&& d_I\left(\int_{\Omega}J(x-y)\underline{I}(y)dy-\underline{I}(x)\right)
+(\beta(x)-\gamma(x))\underline{I}
-\frac{d_S\beta(x)\underline{I}^2}{d_SS_*(x)+(d_S-d_I)\underline{I}}\\
&=& \mu_p^n(d_I)\varepsilon\varphi_n+(m(x)-m_n(x))\varepsilon\varphi_n-
\frac{d_S\beta(x)\varepsilon^2\varphi_n^2}{d_SS_*(x)+(d_S-d_I)
\varepsilon\varphi_n}\\
&\ge& \frac12\mu_p(d_I)\varepsilon\varphi_n-
\frac{d_S\beta(x)\varepsilon^2\varphi_n^2}{d_SS_*(x)+(d_S-d_I)
\varepsilon\varphi_n}\ge0
\end{eqnarray*}
for sufficiently small $\varepsilon$ and $\overline{I}(x)$ satisfies
\begin{eqnarray*}
&& d_I\left(\int_{\Omega}J(x-y)\overline{I}(y)dy-\overline{I}(x)\right)
+(\beta(x)-\gamma(x))\overline{I}
-\frac{d_S\beta(x)\overline{I}^2}{d_SS_*(x)+(d_S-d_I)\overline{I}}\\
&=& -\Lambda(x)-\gamma(x)\overline{I}<0.
\end{eqnarray*}
Now, choose $\varepsilon$ small enough such that $\varepsilon\varphi_n(x)\le\frac{d_S}{d_I}S_*(x)$. Then, the comparison principle implies that \eqref{403} admits a positive solution $I^*(x)$ satisfying $\varepsilon\varphi_n(x)\le I^*(x)\le\frac{d_S}{d_I}S_*(x)$. Meanwhile, denote
\begin{equation*}
S^*(x):=\frac{d_SS_*(x)-d_II^*(x)}{d_S}.
\end{equation*}
Then, $(S^*(x), I^*(x))$ is a pair of nonnegative solution of \eqref{401}.

Further, we claim that $S^*(x)>0$ and $I^*(x)<\frac{d_S}{d_I}S_*(x)$. Otherwise, $I^*(x_0)=\frac{d_S}{d_I}S_*(x_0)$ for some $x_0\in Int(\Omega)$. It is immediately that $S^*(x_0)=0$. Then, by the second equation of \eqref{401}, we obtain that
\begin{equation*}
0\ge d_I\left(\int_{\Omega}J(x-y)I(y)dy-I(x_0)
\right)=\gamma(x_0)I(x_0)>0,
\end{equation*}
which is a contradiction. Thus, $I^*(x)<\frac{d_I}{d_S}S_*(x)$ for all $x\in Int(\Omega)$. If $I^*(x)\in C(\bar{\Omega})$, we can choose a point sequence $\{x_n\}\subset Int(\Omega)$ such that $x_n\to x_0$ for $x_0\in\partial\Omega$, then the same discussion can obtain that $I^*(x)<\frac{d_I}{d_S}S_*(x)$ for all $x\in\partial\Omega$. Meanwhile, we have $S^*(x)>0$. In fact, it can be shown that $S^*(x),I^*(x)\in C(\bar{\Omega})$. Define a functional as follows
\begin{equation*}
F(x,u)=d_I\int_{\Omega}J(x-y)I^*(y)dy-d_Iu+f(x,u),
\end{equation*}
in which
\begin{equation}\label{d}
f(x,u)=\left[
\beta(x)-\gamma(x)-\frac{d_S\beta(x)u}
{d_SS_*(x)+(d_S-d_I)u}\right]u.
\end{equation}
Hence, we have
\begin{equation*}
F_u(x,u)=-d_I+\beta(x)-\gamma(x)-\frac{d_S\beta(x)u}
{d_SS_*(x)+(d_S-d_I)u}-\frac{d_S^2S_*(x)\beta(x)u}
{[d_SS_*(x)+(d_S-d_I)u]^2}.
\end{equation*}
Obviously, $F(x,I^*)=0$ and
\begin{equation*}
F_u(x,I^*)=-d_I\int_{\Omega}J(x-y)\frac{I^*(y)}{I^*(x)}dy-\frac{d_S^2S_*(x)\beta(x)I^*(x)}
{[d_SS_*(x)+(d_S-d_I)I^*(x)]^2}<0.
\end{equation*}
Consequently, the Implicit Function Theorem implies that $I^*(x)\in C(\bar{\Omega})$. By the relation between $S^*(x)$ and $I^*(x)$, it is obvious that $S^*(x)\in C(\bar{\Omega})$.

Next, we intend to prove the uniqueness of positive solutions of \eqref{401} by the method like in \cite{Be2005}. On the contrary, assume that both $I_1(x)$ and $I_2(x)$ are positive solutions of \eqref{401}. Define
\begin{equation*}
\tau^*=\inf\{x\in\bar{\Omega}|~I_1(x)\le\tau^*I_2(x)\}.
\end{equation*}
Note that, by the boundedness of $I_1(x)$ and $I_2(x)$, $\tau^*$ is well-defined. We claim that $\tau^*\le1$. Otherwise, $\tau^*>1$. Then, we have
\begin{eqnarray}\label{404}
&&d_I\left(\int_{\Omega}J(x-y)\tau^*I_2(y)dy-\tau^*I_2(x)\right)\\
&&+
\left[\beta(x)-\gamma(x)-\frac{d_S\beta(x)\tau^*I_2}
{d_SS_*+(d_S-d_I)\tau^*I_2}\right]\tau^*I_2 \nonumber\\
&=& \left
[\frac{d_S\beta(x)I_2}{d_SS_*+(d_S-d_I)I_2}
-\frac{d_S\beta(x)\tau^*I_2}{d_SS_*+(d_S-d_I)\tau^*I_2}
\right]\tau^*I_2<0.
\end{eqnarray}
By the definition of $\tau^*$, there is some $x_*\in\bar{\Omega}$ such that $I_1(x_*)=\tau^*I_2(x_*)$. Thus, the direct computation yields that
\begin{eqnarray}\label{405}
&&d_I\left(\int_{\Omega}J(x_*-y)\tau^*I_2(y)dy-\tau^*I_2(x_*)\right)
\nonumber\\
&&+\left[\beta(x_*)-\gamma(x_*)-\frac{d_S\beta(x_*)\tau^*I_2(x_*)}
{d_SS_*+(d_S-d_I)\tau^*I_2(x_*)}\right]\tau^*I_2(x_*)\\
&=& d_I\int_{\Omega}J(x_*-y)(\tau^*I_2(y)-I_1(y))dy\ge 0.\nonumber
\end{eqnarray}
Now, set $w(x)=\tau^*I_2(x)-I_1(x)$.
Combining \eqref{404} and \eqref{405}, we get $d_I\int_{\Omega}J(x_*-y)w(y)dy=0$, which implies that $w(y)\equiv0$ for $y\in\Omega$. That is, $I_1(x)=\tau^*I_2(x)$ for all $x\in\Omega$. Thus,
\begin{equation*}
\begin{aligned}
0&=d_I\left(\int_{\Omega}J(x-y)I_1(y)dy-I_1(x)\right)+f(x,I_1(x))\\
&=f(x,\tau^*I_2(x))-\tau^*f(x,I_2(x))<0,
\end{aligned}
\end{equation*}
where $f(x,u)$ is defined as \eqref{d}. This is a contradiction. Hence, we have $I_1(x)\le I_2(x)$ for all $x\in\Omega$. On the other hand, the same arguments as above can obtain that $I_1(x)\ge I_2(x)$. Thus, $I_1(x)=I_2(x)$. The uniqueness of positive solutions of \eqref{403} is obtained. And this implies that the positive solutions of \eqref{401} are unique. We finish the proof.
\end{proof}

Now, we give the global stability of positive stationary solutions of \eqref{101} in the sense of Definition \ref{def2}. Before this, a nonlocal dispersal Dirichlet problem is involved. That is, we first consider the problem
\begin{equation}\label{407}
\begin{cases}
\frac{\partial u}{\partial t}=\int_{\mathbb{R}^N}J(x-y)
(u(y,t)-u(x,t))dy+b(x)u-a(x)u^2,~&x\in\Omega, t>0,\\
u(x,0)=u_0(x),~&x\in\Omega,\\
u(x,t)=0,~&x\in\mathbb{R}^N\backslash\Omega, t>0,
\end{cases}
\end{equation}
where $u_0(x)$ is a bounded continuous function. Define
\begin{equation*}
\lambda_p=\inf_{\scriptstyle \varphi\in L^2(\Omega) \atop\scriptstyle
\varphi\neq0}\frac{\int_{\Omega}\varphi^2(x)dx-\int_{\Omega}\int_{\Omega}
J(x-y)\varphi(y)\varphi(x)dydx-\int_{\Omega}b(x)
\varphi^2(x)dx}{\int_{\Omega}\varphi^2(x)dx}.
\end{equation*}
Then, we have the following result.
\begin{lemma}\label{lemma402}
Assume $b(x),a(x)\in C(\bar{\Omega})$ and $a(x)>0$. If $\lambda_p<0$, then \eqref{407} admits a unique positive stationary solution $\tilde{u}(x)\in C(\bar{\Omega})$. Moreover, $\tilde{u}(x)$ is globally asymptotically stable.
\end{lemma}

The proof of Lemma \ref{lemma402} can be found in \cite{Coville2010,SLW2015,MR20091}. We omit it here.

\begin{theorem}\label{theorem403}
Suppose $R_0>1$ and $d_S=d_I=d$. Then, the positive solutions of \eqref{101} converge to the endemic equilibrium $(S^*(x),I^*(x))$ uniformly on $x$ as $t\to+\infty$.
\end{theorem}
\begin{proof}
Let $V(x,t)=\tilde{S}(x,t)+\tilde{I}(x,t)$. Then, according to problem \eqref{101}, we have
\begin{equation}\label{408}
\begin{cases}
\frac{\partial V(x,t)}{\partial t}=d\int_{\mathbb{R}^N}J(x-y)(V(y,t)-V(x,t))dy+\Lambda(x),~&x\in\Omega, t>0,\\
V(x,0)=V_0(x)\ge0,~&x\in\Omega,\\
V(x,t)=0,~&x\in\mathbb{R}^N\backslash\Omega, t>0.
\end{cases}
\end{equation}
Moreover, $\tilde{I}(x,t)$ satisfies
\begin{equation}\label{409}
\frac{\partial I(x,t)}{\partial t}=d\int_{\mathbb{R}^N}J(x-y)(I(y,t)-I(x,t))dy
+(\beta(x)-\gamma(x))I
-\frac{\beta(x)}{V}I^2
\end{equation}
for $x\in\Omega$ and $t>0$. Now, set $H(x,t)=V(x,t)-S_*(x)$. Then, $H(x,t)$ satisfies
\begin{equation*}
\begin{cases}
\frac{\partial H(x,t)}{\partial t}=d\int_{\mathbb{R}^N}J(x-y)(H(y,t)-H(x,t))dy,~& x\in\Omega, ~t>0,\\
H(x,t)=0,~& x\in\mathbb{R}^N\backslash\Omega, ~t>0
\end{cases}
\end{equation*}
with $H(x,0)=(V_0(x)-S_*(x))\in C(\bar{\Omega})$. Following Theorem 2.5 from \cite{AMRT2010}, there is
\begin{equation*}
\|H(\cdot,t)\|_{L^\infty(\Omega)}\le Ce^{-\lambda_1t}
\end{equation*}
for some positive constant $C$ and $\lambda_1$
is defined as in Lemma \ref{lemma302}. Since $H(\cdot,0)\in C(\bar{\Omega})$, we know $H(\cdot,t)\in C(\bar{\Omega})$. Thus,
\begin{equation*}
\|H(\cdot,t)\|_{C(\bar{\Omega})}\le C_*e^{-\lambda_1t}
\end{equation*}
for some positive constant $C_*$. Hence, we have $V(x,t)\to S_*(x)$ uniformly on $\bar{\Omega}$ as $t\to+\infty$. For any $\varepsilon>0$, there exists some $T>0$ such that
\begin{equation*}
S_*(x)-\varepsilon\le V(x,t)\le S_*(x)+\varepsilon
\end{equation*}
for all $t\ge T$. Below, we consider the following two auxiliary problems
\begin{equation}\label{410}
\begin{cases}
\frac{\partial\overline{I}}{\partial t}=d\left(\int_{\Omega}J(x-y)\overline{I}(y,t)dy-\overline{I}(x,t)\right)
+(\beta(x)-\gamma(x))\overline{I}
-\frac{\beta(x)\overline{I}^2}{S_*(x)+\varepsilon},~&x\in\Omega, t>T,\\
\overline{I}(x,t)=I(x,T)>0, ~&x\in\Omega
\end{cases}
\end{equation}
and
\begin{equation}\label{411}
\begin{cases}
\frac{\partial\underline{I}}{\partial t}=d\left(\int_{\Omega}J(x-y)\underline{I}(y,t)dy-\underline{I}(x,t)\right)
+(\beta(x)-\gamma(x))\underline{I}
-\frac{\beta(x)\underline{I}^2}{S_*(x)-\varepsilon},~&x\in\Omega, t>T,\\
\underline{I}(x,t)=I(x,T)>0, ~&x\in\Omega.
\end{cases}
\end{equation}
By comparison principle, we get
\begin{equation*}
\underline{I}(x,t)\le \tilde{I}(x,t)\le \overline{I}(x,t),~~~x\in\Omega,~t>T.
\end{equation*}
Now, using Lemma \ref{lemma402}, we have
\begin{equation*}
\underline{I}(x,t)\to\underline{I}_\varepsilon(x)~\text{and}~
\overline{I}(x,t)\to\overline{I}_\varepsilon(x)~~\text{uniformly on}~\bar{\Omega}~\text{as}~t\to+\infty,
\end{equation*}
where $\overline{I}_\varepsilon(x)$ and $\underline{I}_\varepsilon(x)$ are positive stationary solutions of \eqref{410} and \eqref{411}, respectively. Moreover, we claim that $\overline{I}_\varepsilon(x)$ and $\underline{I}_\varepsilon(x)$ are both monotone with respect to $\varepsilon$. Indeed, for any $\varepsilon_1, \varepsilon_2$ with $\varepsilon_1>\varepsilon_2$ and $S_*(x)>\max\{\varepsilon_1, \varepsilon_2\}$, letting $\underline{I}_{\varepsilon_1}(x)$ and $\underline{I}_{\varepsilon_2}(x)$ are respectively the positive stationary solutions of \eqref{411} corresponding to $\varepsilon=\varepsilon_1$ and $\varepsilon=\varepsilon_2$. Note that $\underline{I}_{\varepsilon_1}(x)$ satisfies
\begin{eqnarray*}
&& d\left(\int_{\Omega}J(x-y)\underline{I}_{\varepsilon_1}(y)dy
-\underline{I}_{\varepsilon_1}(x)\right)
+(\beta(x)-\gamma(x))\underline{I}_{\varepsilon_1}
-\frac{\beta(x)\underline{I}_{\varepsilon_1}^{2}}{S_*(x)-\varepsilon_2}\\
&=& \frac{\beta(x)\underline{I}_{\varepsilon_1}^{2}}
{S_*(x)-\varepsilon_1} -\frac{\beta(x)\underline{I}_{\varepsilon_1}^{2}}
{S_*(x)-\varepsilon_2}>0.
\end{eqnarray*}
On the other hand, it is easy to see that $MS_*(x)$ is a super solution of problem
\begin{equation}\label{4401}
d\left(\int_{\Omega}J(x-y)u(y)dy-u(x)\right)+(\beta(x)-\gamma(x))u(x)
-\frac{\beta(x)u^2(x)}{S_*(x)-\varepsilon_2}=0~~\text{in}~\Omega,
\end{equation}
provided the constant $M$ large enough. Here, $M$ dose not depend on $\varepsilon$. Meanwhile, the positive solution of \eqref{4401} has a positive lower bound by the same proof method in Theorem \ref{theorem401}, denoted by $\delta_0$. Then, by the uniqueness of positive stationary solutions of \eqref{411} for each fixed $\varepsilon$, we have
\begin{equation*}
\delta_0\le\underline{I}_{\varepsilon_1}(x)<\underline{I}_{\varepsilon_2}(x)\le MS_*(x).
\end{equation*}
That is, $\underline{I}_{\varepsilon}(x)$ is strictly decreasing and uniformly bounded on $\varepsilon$. Meanwhile, the same arguments can obtain that $\overline{I}_{\varepsilon}(x)$ is strictly increasing and uniformly bounded on $\varepsilon$. Thus, there exists a sequence $\{\varepsilon_n\}$ satisfying $\varepsilon_n\to0$ as $n\to+\infty$ such that
\begin{equation*}
\underline{I}_{\varepsilon_n}(x)\to I_1(x)~\text{and}~\overline{I}_{\varepsilon_n}(x)\to I_2(x)~\text{uniformly on}~\bar{\Omega}~\text{as}~n\to+\infty
\end{equation*}
for some positive functions $I_i(x)$ $(i=1,2)$ which are the positive solutions of
\begin{equation}\label{412}
d\left(\int_{\Omega}J(x-y)v(y)dy-v(x)\right)+(\beta(x)-\gamma(x))v(x)
-\frac{\beta(x)v^2(x)}{S_*(x)}=0~~\text{in}~\Omega.
\end{equation}
Consequently, the uniqueness of positive solutions of \eqref{412} implies that $I_1(x)=I_2(x):=I^*(x)$. Then, we get $\tilde{I}(x,t)\to I^*(x)$ uniformly on $\bar{\Omega}$ as $t\to+\infty$. Additionally, it is obvious that
$S_*(x)=S^*(x)+I^*(x)$. Thus, we have $\tilde{S}(x,t)\to S^*(x)$
uniformly on $\bar{\Omega}$ as $t\to+\infty$. The  proof is complete.
\end{proof}

\begin{remark}{\rm
In Theorem \ref{theorem403}, we assume $d_S=d_I$. It is pointed out that if $d_S\neq d_I$, the case becomes very complicate and we leave it for the further study.
}
\end{remark}
\section{Discussion}

\noindent

This paper is concerned with a nonlocal dispersal SIS epidemic model. We have discussed the existence, uniqueness and stability of the disease-free equilibrium and the endemic equilibrium.
From a biological point of view, Theorem \ref{theorem302} implies that the disease will be extinct as $R_0<1$, and Theorem \ref{theorem403} implies that the disease persists when $R_0>1$.

Additionally, in biological sense, we know that $x$ is a low-risk site if the disease transmission rate $\beta(x)$ is lower than the local recovery $\gamma(x)$ and is a high-risk site if $\beta(x)>\gamma(x)$. Meanwhile, $\Omega$ is a low-risk domain if $\int_{\Omega}\beta(x)dx<\int_{\Omega}\gamma(x)dx$ and is a high-risk domain if $\int_{\Omega}\beta(x)dx\ge\int_{\Omega}\gamma(x)dx$, see \cite{ABL-2008}.
Seen from Corollary \ref{cor2}, the disease will persist as long as there are some high-risk site and the movements of the infected individuals are slow, but the disease will be extinct if the movements of the infected individuals are quick. That is to say, in this case, the nonlocal dispersal of the infected individuals will speed up the extinction of the disease. Meanwhile, if the habitat of the population is filled with the low-risk sites, then the disease will be extinct no matter what the type of the dispersal of the population is. Particularly, the
disease may be extinct as long as the movements of the infected individuals are fast even though the species live in a high-risk domain. This implies that the nonlocal dispersal of the infected individuals may suppress the spread of the disease in a high-risk domain. We hope that these results will be useful for the disease control.

\section*{Acknowledgments}

\noindent

Fei-Ying Yang was partially supported by NSF of China (11401277) and Wan-Tong Li was
partially supported by NSF of China (11271172).

\end{document}